\documentclass[12pt]{amsart}
\usepackage[letterpaper,margin=1in]{geometry}

\usepackage{parskip}
\usepackage{amsmath}
\usepackage{amsthm}
\usepackage{amssymb}
\usepackage{mathtools}
\usepackage{bbm}
\usepackage[dvipsnames]{xcolor}
\usepackage{dsfont}
\usepackage{dsfont}
\usepackage{graphicx}
\usepackage{caption}
\usepackage[labelfont=]{subcaption}
\usepackage{ragged2e}
\usepackage{hyperref}
\usepackage{tikz}
\usetikzlibrary{calc}

\newcounter{todo}

\theoremstyle{plain}
\newtheorem{theorem}{Theorem}
\newtheorem{lemma}[theorem]{Lemma}
\newtheorem{corollary}[theorem]{Corollary}

\theoremstyle{definition}
\newtheorem{definition}[theorem]{Definition}

\newtheorem{assumption}[theorem]{Assumption}

\theoremstyle{remark}
\newtheorem{remark}[theorem]{Remark}

\newcounter{locallabel}[theorem]

\newcommand*\E{\mathbb{E}}
\newcommand*\N{\mathbb{N}}
\renewcommand*\P{\mathbb{P}}
\newcommand*\R{\mathbb{R}}  
\newcommand*\Z{\mathbb{Z}}
\newcommand*\FF{\mathcal{F}}

\newcommand*\PP{\mathcal{P}}
\newcommand*\ep{\varepsilon}
\newcommand*\La{\Lambda}

\newcommand*\om{\omega}
\newcommand*\dR{{\partial R}}
\newcommand*\dRn{{\partial R_n}}
\newcommand{\wt}{\widetilde}
\newcommand*\tlh{\wt h}
\newcommand*\tlp{\wt p}

\newcommand*\defn[1]{\textit{#1}}

\newcommand*\restr[1]{\vert_{#1}}
\newcommand*\dP{\partial}

\DeclareMathOperator{\diam}{diam}

\DeclarePairedDelimiter\abs{\lvert}{\rvert}

\let\oldtexttt\texttt
\renewcommand*\texttt[1]{\textnormal{\oldtexttt{#1}}}

\mathtoolsset{showonlyrefs}

\newcommand*\MRh{M(R; h_{R'})}
\newcommand*\MRth{M(R; \tlh_{R'})}
\newcommand*\MRdR{M(R; h_{\dP R})}
\newcommand*\MRn{M(R; h_{\dP R_n})}

\newcommand*\muMRh{\mu_{\MRh}}
\newcommand*\muMRth{\mu_{\MRth}}
\newcommand*\muMRdR{\mu_{\MRdR}}
\newcommand*\muMRn{\mu_{\MRn}}

\begin{document}

\title[The concentration inequality for a perturbed model]{%
The concentration inequality for a discrete height function model
perturbed by random potential}

\author{Andrew Krieger}
\address{Department of Mathematics, University of California, Los Angeles}
\email{akrieger@math.ucla.edu}


\date{\today}

\begin{abstract}
We present a proof of the concentration inequality
for a discrete random surface model,
where the underlying potential is perturbed by an additive random potential.
The proof is based on annealing the random potential,
and follows the method of~\cite{CEP96} and other works.
Our result demonstrates the robustness of this method.
\end{abstract}

\maketitle

\section{Introduction} \label{s_intro}

Random surface models are common objects of study in statistical physics
and combinatorics (e.g.~\cite{Geo88,She05}).
Random surface models describe membrane surfaces or hypersurfaces
whose formation and shape are governed on the microscopic scale
by a random interaction or process.
One is interested in predicting the macroscopic behavior
using data about the microscopic random process.
There is a wide literature studying many different random surface models.
A incomplete list includes
domino tilings and dimer models (e.g.~\cite{Kas63,CEP96,CKP01}),
polymer models (e.g.~\cite{BiPr18,BeYa19}),
lozenge tilings (e.g.~\cite{Des98,LRS01,Wil04}),
Ginzburg-Landau models (e.g.~\cite{DeGiIo00,FuOs04}),
the Ising model (e.g.~\cite{DKS92,Cer06}),
asymmetric exclusion processes (e.g.~\cite{FS06}),
sandpile models (e.g.\cite{LP08}),
the six vertex model (e.g.~\cite{BCG2016,CoSp16,ReSr16}),
and Young tableaux (e.g.~\cite{LS77,VK77,PR07}).
Most relevant to the current work are
discrete Lipschitz functions (e.g.~\cite{PSY13,Pel17,LT20})
and graph homomorphisms from certain discrete graphs to $\Z$
(e.g.~\cite{BHM00,Kah01,Gal03,CPST20}).

The specific model studied here
is built upon graph homomorphisms
from subgraphs of the nearest-neighbor lattice $\Z^m$ ($m \ge 1$)
into the graph $\Z$, again with nearest-neighbor edge structure;
in the sequel these are called ``height functions.''
In dimension $m=2$ these $\Z$-valued graph homomorphisms
correspond to height functions of the six-vertex model,
under the convention used in~\cite{Sri16}.
The model is extended by assigning random weights to all such graph
homomorphisms depending on the height values they take
(see the definition of the perturbed measures~$\mu(\cdot, \omega)$
in Definition~\ref{d_meas} below).
Homogenization of this model was previously studied by the current author
in conjunction with coauthors in~\cite{KMT21+},
and the same group of authors studied the underlying model
(i.e.\ $\Z$-valued graph homomorphisms from subsets of $\Z^m$,
without random perturbation) in~\cite{KMT20}.

The main result of this note is the concentration inequality,
presented as Theorem~\ref{thm_conc} below.
Our result implies that the probability mass of a
perturbed random distribution~$\mu(\cdot, \omega)$
concentrates on functions that are ``close,''
relative to the size of the domain,
to the $\mu(\cdot, \omega)$-expected height at each point in the domain.
To properly ground this result in the surrounding literature,
let us consider analogues of the concentration inequality in other settings.
Generally, concentration inequalities control
the magnitude and likelihood of fluctuations of a random variable,
especially a random statistic that quantifies some large-scale property
of an underlying random system.
The concentration inequality is a bound on the probability that the random
variable differs from its expected value (or sometimes, another typical value)
by more than $t > 0$ (in some metric).
The probability bound vanishes as $t \to \infty$
or as the size of the underlying random system goes to infinity.
The simplest example is the weak law of large numbers
(cf.\ \cite[Theorem~2.2.3]{Durrett}):
Let $X_1,\dotsc,X_n$ be uncorrelated random variables
(meaning that $E(X_iX_j) = E(X_i) E(X_j)$ for $i \ne j$),
all with common mean $\mu \in \R$ and
common variance bound $E[(X_i-E(X_i))^2] \le \sigma^2$ for some $\sigma^2 > 0$.
Define $M_n = (X_1 + \dotsb + X_n) / n$.
Then $E(M_n) = \mu$ and for any $t > 0$,
\[
	P \biggl( | M_n - \mu | > \frac{t}{\sqrt n} \biggr)
	\le \frac{\sigma^2}{t^2} .
\]

Focusing on random surface models, concentration inequalities are closely
related to the limit shape phenomenon.
A limit shape occurs when samples of the random
microscopic model converge with high probability to a deterministic shape, i.e.\
the limit shape, under a suitable scaling limit.
The concentration inequality provides a route towards proving
the existence of a limit shape.
Indeed, the concentration inequality implies that when the system size is large,
a sample of the model is close to the expected value for that system size,
with high probability; see for example Figure~\ref{fig}.
Then it remains to show that these expected values converge
as the system size tends to infinity,
and the limit is necessarily the limit shape for the model.
This convergence can be established by means of a variational principle,
which characterizes the limit shape via a variational problem.
See e.g.~\cite[Theorem~2.12]{MT20}, which is the variational principle for a
similar model, and whose proof is based on a concentration inequality
\cite[Theorem~3.10]{MT20}.
Let us also mention the concentration inequality~\cite[Theorem~21]{CEP96},
whose proof is the inspiration for the work that follows.
Finally we cite Chapter~7 of~\cite{She05}.
There, a large deviations principle is proved
for a wide class of random surface models,
and \cite{She05} also explains how concentration inequalities can be derived
from that large deviations principle under certain conditions
(i.e.\ uniqueness of the minimizer of the rate function).

\begin{figure}
	\centering
	\begin{subfigure}{0.24\textwidth}
		\includegraphics[width=\textwidth]{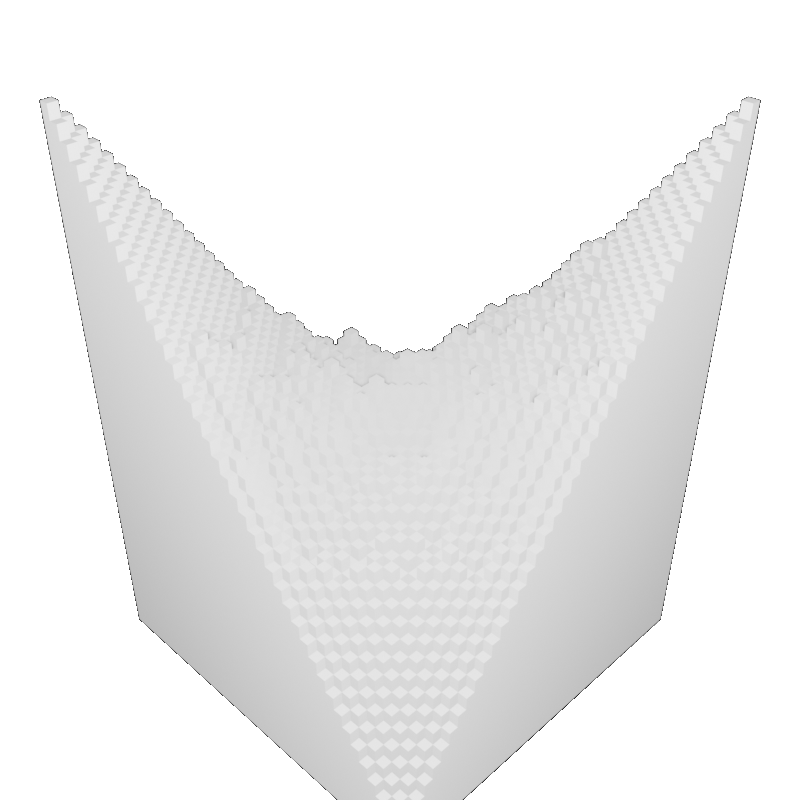}
		\caption{$25 \times 25$ system.}
		\label{figa}
	\end{subfigure}
	\begin{subfigure}{0.24\textwidth}
		\includegraphics[width=\textwidth]{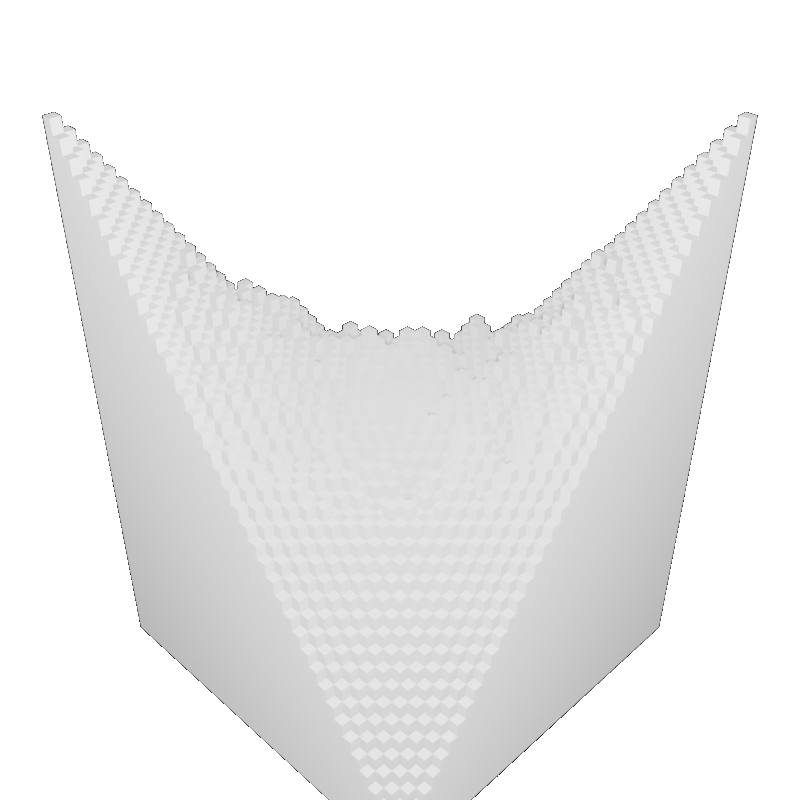}
		\caption{$25 \times 25$ system.}
		\label{figb}
	\end{subfigure}
	\begin{subfigure}{0.24\textwidth}
		\includegraphics[width=\textwidth]{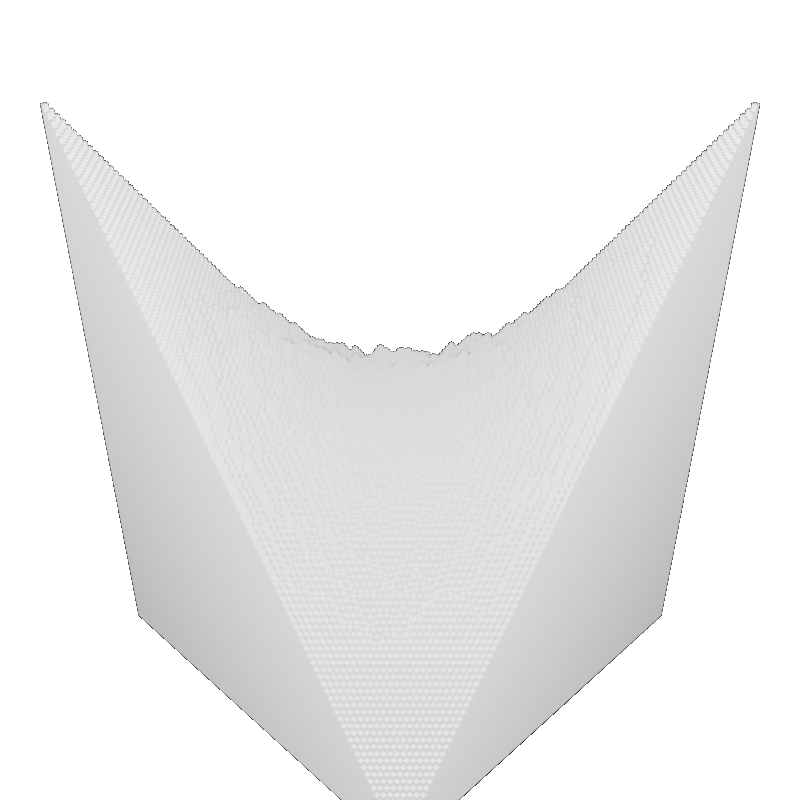}
		\caption{$100 \times 100$ system.}
		\label{figc}
	\end{subfigure}
	\begin{subfigure}{0.24\textwidth}
		\includegraphics[width=\textwidth]{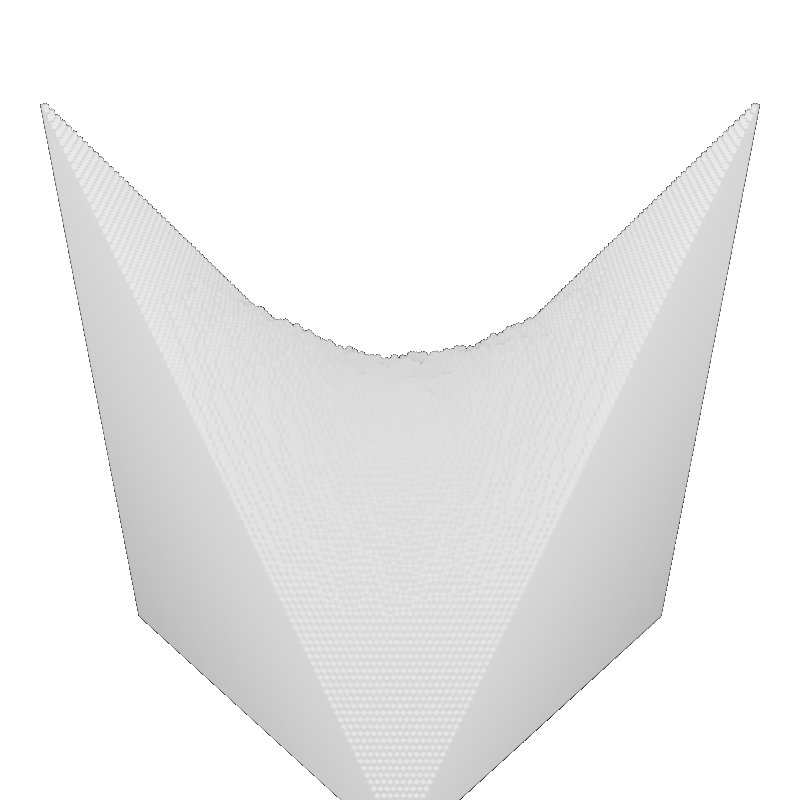}
		\caption{$100 \times 100$ system.}
		\label{figd}
	\end{subfigure}
\caption{
Pictured are 3D plots of height functions on $\Z^2$.
The domains are respectively $25 \times 25$ and $100 \times 100$ grids,
realized as square faces, and the height values assigned by the height function
are indicated by the height of the column above each square face.
Boundary values are fixed, and are chosen so as to increase or decrease
as rapidly as possible along each boundary edge.
It is visually apparent that each sample is close to a smooth limiting shape,
except for small fluctuations, and that on the larger domain,
the fluctuations are smaller relative to the domain size.
This illustrates the conclusion of the concentration inequality.
}
\label{fig}
\end{figure}

Now, let us describe our strategy for proving the homogenized concentration
inequality.
For the simpler case of random surface models without random perturbations,
there are two approaches in the literature
to proving the concentration inequality directly:
an approach that leverages monotonicity
and stochastic dominance (e.g.~\cite{CEP96,LT20})
and a dynamic approach using a natural Markov chain
on height functions (e.g.~\cite{MT20}).
The second approach does not easily carry over to the case
of random perturbations.
It relies on monotonicity on the quenched level
(i.e.\ with fixed random potential~$\om$),
but monotonicity only occurs in the annealed setting
(i.e.\ after averaging over~$\om$).
This complicates the dynamic structure of the second approach.
Therefore we use the first approach.
See the statement and proof of Lemma~\ref{lem_18} for details.

The remainder of this note is divided into two sections.
Section~\ref{s_defs} defines the model in detail,
and Section~\ref{s_proof} provides the exact statement
proof of the concentration inequality.

\section{Definitions} \label{s_defs}

The notation and definitions below are similar to those in~\cite{KMT21+}.
For further explanation and further examples,
we kindly refer the reader to~\cite[Section~2]{KMT21+}.

Throughout the sequel, whenever $\Z^m$ is viewed as a graph,
the edge set is always taken to be the collection of all nearest neighbor edges;
in other words, two points $x,y \in \Z^m$ are adjacent if and only if
the $\ell^1$ distance $|x-y|_1$ is exactly $1$.
Likewise, subsets of $R \subset \Z^m$ are taken to be graphs
with the induced subgraph structure,
so any two vertices $x,y \in \La \subset \Z^m$ are adjacent
if and only if $|x-y|_1 = 1$.
We assume that the subgraph $R \subset \Z^m$ is connected.
The edges of these graphs are written as $e_{x,y}$ for the edge connecting
$x,y \in \Z^m$.
These are undirected edges, i.e.\ $e_{y,x} = e_{x,y}$.

\begin{definition}[Height function] \label{d_ht_func}
A \defn{height function} on a connected subgraph $R \subset \Z^m$
is a parity-preserving graph homomorphism $h_R: R \to \Z$.
In other words, if $x, y \in R$ and $x \sim y$,
then $\abs{ h_R(x) - h_R(y) } = 1$,
and for any $x = (x_1, \dotsc, x_m) \in R$,
\begin{equation} \label{e_ht_func_parity}
	h_R(x) = \sum_{i=1}^m x_i \pmod 2 .
\end{equation}

The set of all height functions on $R$ is denoted as
\begin{equation}
	M(R)
	:= \bigl\{ h_R: R \to \Z
	\, \big| \, \text{$h_R$ is a height function} \bigr\}
\end{equation}
and given a subgraph $R' \subset R$ and a height function
$h_{R'} \in M(R')$,
the set of height functions extending $h_{R'}$ to $R$ is denoted as
\begin{equation}
	M(R; h_{R'})
	:= \bigl\{ h_R \in M(R)
	\, \big| \, h_R|_{R'} = h_{R'} \bigr\} .
\end{equation}
\end{definition}

The set $M(R)$ is never empty; indeed, the function $h_R$ that maps each point
in $R$ to either $0$ or to $1$ according to the parity of the source point
is a height function.
A natural question to ask is under which conditions the set of extensions
$M(R; h_{R'})$ is nonempty. Necessary and sufficient conditions are given by
the Kirszbraun theorem for functions on $\Z^d$,
which is a discrete analogue of the
classical Kirszbraun theorem for Lipschitz functions on $\R^d$.
This is a well-known result (see e.g.\ \cite[Lemma 4.3.1]{She05}),
and we omit the proof from this article.

\begin{theorem}[Kirszbraun theorem for~$\Z^d$] \label{p_kirszbraun}
Let~$R$ be a connected region of $\Z^m$, let~$R'$ be a subset of~$R$,
and let~$h': R' \to \Z$ be a graph homomorphism that preserves parity.
There exists a graph homomorphism~$h: R \to \Z$
such that~$h = h'$ on~$R'$ if and only if for all~$x, y \in R'$,
\begin{equation} \label{e_kirszbraun}
	\bigl| h'(x) - h'(y) \bigr| \le d_R(x, y),
\end{equation}
where~$d_R$ denotes the graph distance on~$R$.
\end{theorem}

\begin{remark}
One might hope to derive sufficient conditions that are easier to verify,
such as the following: suppose that $R'$ is a line segment,
i.e.\ $R' = \{v, v+e, v+2e, \dotsc, v+(\ell-1)e\}$
for some $\ell \in \N$ and $v,e \in \Z^d$ with $|e|_1 = 1$.
Then $R'$ is an isometric subgraph of $R$,
meaning that for any points $x,y \in R'$, $d_{R'}(x,y) = d_R(x,y)$.
Indeed, clearly $d_{R'}(x,y) \ge d_R(x,y)$,
and for any $i,j$, by observation
$d_{R'}(v+ie, v+je) = |i-j| = d_{\Z^d}(v+ie, v+je) \le d_R(x,y)$.
Since $R'$ is an isometric subgraph of every $R \supseteq R'$,
it follows that for any $h_{R'} \in M(R')$ and any $x,y \in R'$,
we have $|h_{R'}(x) - h_{R'}(y)| \le d_{R'}(x,y) = d_R(x,y)$.
Thus by the Kirszbraun theorem $h_{R'}$ admits at least one extension to $R$.
This idea can be pushed a little bit farther.
For example, the exact same argument works whenever $R'$ is a geodesic,
i.e.\ a shortest path between two points in $R$;
a similar argument via the isometric subgraph property applies
when ever $R'$ is a box,
i.e.\ when $R' = \{(z_1,\dotsc,z_i) \in \Z^d \colon a_i \le z_i \le b_i\}$
for some $a_1 \le b_1, \dotsc, a_d \le b_d$.
However, this approach ultimately is not fruitful for our current purposes.
The motivating example we use is the case where $R$ is a box
and $R'$ is its boundary, and one can check easily that this subgraph $R'$
is not isometric. For example, in two dimensions, if
$R = \{0,1,2\} \times \{0,1,2\}$,
then the opposite midpoints $(0,1)$ and $(2,1)$ have
$d_{R'}((0,1), (2,1)) = 4 > 2 = d_R((0,1), (2,1))$.
Indeed, one can check that there is a unique height function on $R'$
with $h_{R'}((0,1)) = 1$ and $h_{R'}((2,1)) = 5$,
which violates the Kirszbraun hypothesis~\eqref{e_kirszbraun}.
As such, in the sequel we must account for the possibility that
$M(R; h_{R'})$ may be empty for some or even for all $h_{R'} \in M(R')$.
\end{remark}

Moving on to the random potential, the assumptions we make are the same as
in~\cite{KMT21+}, namely:

\begin{assumption}[Random potential~$\omega$]\label{a_omega}
Let $\omega = (\omega_e)_{e \in E(\Z)} \in \R^{E(\Z)}$
denote a random potentials, defined on the set of edges~$E(\Z)$ of~$\Z$,
that satisfies the following assumptions:
\begin{itemize}
	\item $\omega$ is almost surely finite, in the sense that
	\begin{align}\label{e_omega_bdd}
		 C_\omega
		 := 1 \vee \sup_{e \in E(\Z)} \abs{ \omega_{e} }
		 \in L^1
		 \qquad \text{(i.e.\ $\mathbb{E}[C_\omega] < \infty$)}.
	\end{align}
\item $\omega$ is shift invariant, i.e.\
	for any $k \in \N$, any edges
	$e_{x_1,y_1}, \dotsc, e_{x_k,y_k} \in E(\Z)$,
	any~$z \in \Z$, and any bounded and measurable function
	$\xi: \R^k \to \R$,
	\begin{equation}\label{e_shift_invariance}
		\E \bigl[
			\xi(\om_{e_{x_1,y_1}} , \dotsc, \om_{e_{x_k,y_k}})
		\bigr]
		= \E \bigl[
			\xi(\om_{e_{x_1+z,y_1+z}}, \dotsc,
				\om_{e_{x_k+z,y_k+z}})
		 \bigr] .
	\end{equation}
\item Moreover, the random potential~$\omega$ is ergodic
	with respect to the set of shifts
	$\{\tau_z \,|\, z \in \Z, \, z = 0 \pmod 2 \}$,
	where $\tau_z(e_{x,y}) = e_{x+z,y+z}$.
	This means that if $E = \tau_2^{-1}(E)$,
	then $\P(E) \in \{0, 1\}$.
\item As a matter of normalization, assume that
	$ \E [ \omega_{(0,1)} ] = 0$.
\end{itemize}
\end{assumption}

\begin{definition}[The Hamiltonian]
Given a finite connected graph $R \subset \Z^m$,
the interior Hamiltonian $H^\circ_R: M(R) \to \R$
and the exterior Hamiltonian $H^+_R: M(R^+) \to \R$
are defined as
\begin{equation} \label{e_def_hamiltonian}
	H^\circ_R(h_R, \omega) = \sum_{e \in E(R)} \omega_{h_R(e)}
	\qquad \text{and} \qquad
	H^+_R(h_{R^+}, \omega) = \sum_{e \in E(R^+)} \omega_{h_{R^+}(e)} ,
\end{equation}
where $R^+ := R \cup \{x \in \Z^m \colon \exists y \in R, x \sim y\}$
is the extension of the subgraph $R$ by all neighboring vertices in the ambient
graph $\Z^m$.
\end{definition}

\begin{definition}[Quenched measure] \label{d_meas}
Given a realization~$\omega$ of the random potential
and a set $A \subset M(R)$ of height functions,
the partition function $Z_\omega(A)$ is given by
\begin{equation} \label{e_d_partition_function}
	Z_\omega(A)
	= \sum_{h_R \in A} \exp \bigl( H^\circ_R(h_R, \omega) \bigr) \,.
\end{equation}

Given a subgraph~$R' \subset R$,
a height function $h_{R'} \in M(R')$,
and a realization~$\omega$ of the random potential,
the perturbed measure $\muMRh(\cdot, \omega)$
is the probability measure on the finite space~$M(R; h_{R'})$,
defined on atoms~$h_R \in \MRh$ by
\begin{equation}\label{e_d_gibbs_measure}
	\muMRh(h_R, \om)
	:= \frac{\exp \bigl( H^\circ_R(h_R, \omega) \bigr)}
		{Z_\omega \bigl( M(R; h_{R'}) \bigr)} .
\end{equation}
\end{definition}

\begin{definition}[Annealed measure] \label{d_ann}
Below we prove that, for any $R' \subset R$ and any $h_{R'} \in M(R')$,
the function $(\om, A) \mapsto \muMRh(A, \om)$
from~$\Omega \times \PP(M(R; h_{R'}))$ to~$[0,1]$
is a probability kernel.
In other words, for all $\om \in \Omega$,
$A \mapsto \muMRh(A, \om)$ is a probability measure,
and for all $A \subset M(R; h_{R'})$,
$\om \mapsto \muMRh(A, \om)$ is measurable.
Thus, the formula
\[
	(\muMRh \circ \P)(A) := \E[ \muMRh(A, \om) ],
	\quad A \subset M(R; h_{R'})
\]
defines a probability measure~$\muMRh \circ \P$ on~$M(R; h_{R'})$.
Moreover, for any function~$f: M(R; R') \to \R$,
if~$h_R$ is a random variable with law~$\muMRh \circ \P$
then
\begin{equation} \label{e_cond_exp}
	E[f(X)] = \E \bigl[ E_{\muMRh(\cdot, \om)}(f) \bigr] .
\end{equation}
\end{definition}

\begin{proof}[Proof (of claims in Definition~\ref{d_ann})]
By definition~$\muMRh(\cdot, \om)$ is a probability measure for fixed~$\om$.
The second property follows from how $\muMRh(h_R, \om)$ is defined:
the numerator is a sum of finitely many random potential values~$\om_e$
inside of the (continuous, hence measurable) exponential function,
and the denominator is a finite sum over copies of the numerator,
only using different height functions to select the random potential
values~$\om_e$.
The equation~\eqref{e_cond_exp} is a standard identity for regular conditional
distributions and its proof is a straightforward exercise;
see e.g.\ \cite[Exercise 5.1.14]{Durrett}.  
Note that since~$M(R;R')$ is finite,
each function~$f:M(R;R') \to \R$ is measurable.
\end{proof}

\section{The concentration inequality} \label{s_proof}

Having stated the definitions above,
we are prepared to state the main result of this note,
then move on to proofs.

\begin{theorem} \label{thm_conc}
Let $R_n \subset \Z^m$ be a sequence of finite, connected subgraphs
such that $\diam(R_n) := \max_{x,y \in R_n} |x-y|_1 \le An$ for some $A>0$.
Let $\ep > 0$ and $h_\dRn \in M(\dRn)$ be given,
and let $\mu_n = \muMRn(\cdot, \omega)$
denote the (perturbed) distribution on $\MRn$.
Then for any $c > 0$ and any $n \in \N$,
\begin{equation} \label{e_conc}
	\mu_n \circ \P \biggl( \max_{v \in R_n} \,
		\abs[\Big]{ h_{R_n}(v) - E_{\mu_n}(h_{R_n}(v)) }
			\ge c \sqrt{n} \biggr)
	\le 2 \, |R_n| \, e^{-nc^2/A} .
\end{equation}
\end{theorem}

\begin{remark}
Theorem~\ref{thm_conc} gives quantitative bounds for the probability that
a height function~$h_R$ differs from expected value on the scale of~$\sqrt{n}$.
The probability bounds are exponential in a constant times $n$,
which comes about because of the one-dimensional
nature of the Azuma--Hoeffding inequality.
In comparison, results such as the large deviations principle
of~\cite{She05,LT20} achieve a volume-order term in the exponential.
The cost of those results is that control over the size of fluctuations
is not quantitative.
For example, ignoring the random potential~$\omega$
and using the uniform measure on~$M(R; h_{\dP R})$ instead,
the large deviations principle implies that
$\mu_n( \max_v |h_{R_n}(v) - E_{\mu_n} (h_{R_n}(v))| \ge \ep )
	\le \exp(-n^m I(\ep))$.
Here $I(\ep) > 0$ is (related to) the rate function of the LDP,
which does not admit an obvious closed form expression in terms of~$\ep$.
It would be interesting if we could obtain a quantitative concentration result
with volume-order term in the exponential of the probability bound.
\end{remark}

We will build up to the proof of the concentration inequality
via a few intermediate results.
Lemma~\ref{lem_18} below establishes the monotonicity property,
which is the main ingredient of the proof.
We derive from it Corollary~\ref{cor_19},
which is used in the proof of an auxiliary concentration inequality
in Lemma~\ref{lem_conc}.
The difference between Lemma~\ref{lem_conc} and the main theorem
is that the former addresses only a single point $v \in R$,
whereas the latter concerns the maximum deviation from the mean
over the entire domain.
The statements and proofs of these results are based on
the method presented in~\cite{CEP96};
we cite the analogous steps where appropriate below.
Differences arise starting with Corollary~\ref{cor_19} below,
where the shift-invariant and ergodic properties of the law of~$\om$
must be used to account for the fact that height functions
with different base heights ``see'' different random potential values~$\om_e$.
However, the essential steps of the proof still goes through,
since even under the influence of~$\om$ the relevant measures
are Gibbs measures, based upon a finite-range potential
(indeed, a nearest-neighbor potential, which enforces the Lipschitz property).
It seems like the proof should extend to other finite-range models
and perhaps beyond, but for brevity we will not explore that idea further here.

\begin{lemma}[cf.\ {\cite[Lemma 18]{CEP96}}] \label{lem_18}
Let $R' \subset R$ and let $h_{R'}, \tlh_{R'} \in M(R')$
be such that $h_{R'} \le \tlh_{R'}$
and that~$M(R; h_{R'})$ and~$M(R; \tlh_{R'})$ are not empty.
Then for any realization $\om$,
$\muMRh(\cdot, \om)$ is stochastically dominated by $\muMRth(\cdot, \om)$.
More precisely, there exists a measurable function
$\pi: M(R; h_{R'}) \times M(R; \tlh_{R'}) \times \Omega \to [0,1]$
such that:
\begin{itemize}
\item
for almost every $\omega$,
$\pi(\cdot, \cdot, \omega)$ is a coupling, i.e.\
\[
	\sum_{\tlh_R} \pi(h_R, \tlh_R, \omega) = \muMRh(h_R, \om)
	\qquad \text{and} \qquad
	\sum_{h_R} \pi(h_R, \tlh_R, \omega) = \muMRth(\tlh_R, \om) ,
\]
and
\item
$
	\pi \bigl( \{h_R \le \tlh_R \}, \omega \bigr) = 1 .
$
\end{itemize}
\end{lemma}

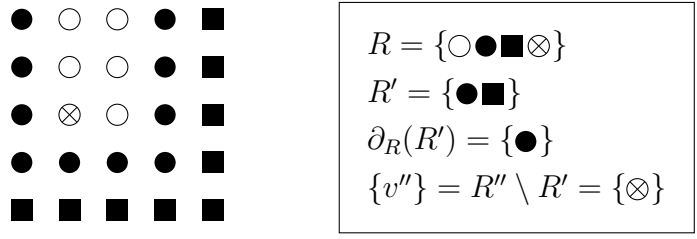
\begin{figure}
\centering
\begin{tikzpicture}[x=0.25in, y=-0.25in]
	\newcommand*\prad{4pt}
	\newcommand*\pfil{\hspace\prad}
	\newcommand*\pext[1]{
		\draw #1 circle [radius=\prad];
		\draw #1 +(45:\prad) -- +(225:\prad);
		\draw #1 +(135:\prad) -- +(315:\prad);}
	\newcommand*\pint[1]{\draw #1 circle [radius=\prad];}
	\newcommand*\prbd[1]{\fill #1 circle [radius=\prad];}
	\newcommand*\pbdi[1]{%
		\fill #1 +(-\prad,-\prad) rectangle +(\prad,\prad);}
	\prbd{(1,1)} \pint{(2,1)} \pint{(3,1)} \prbd{(4,1)} \pbdi{(5,1)}
	\prbd{(1,2)} \pint{(2,2)} \pint{(3,2)} \prbd{(4,2)} \pbdi{(5,2)}
	\prbd{(1,3)} \pext{(2,3)} \pint{(3,3)} \prbd{(4,3)} \pbdi{(5,3)}
	\prbd{(1,4)} \prbd{(2,4)} \prbd{(3,4)} \prbd{(4,4)} \pbdi{(5,4)}
	\pbdi{(1,5)} \pbdi{(2,5)} \pbdi{(3,5)} \pbdi{(4,5)} \pbdi{(5,5)}
	\begin{scope}[shift={(8,1)}]
		\node (r) at (0,0) [anchor=north west] {$R = \{$};
			\pint{(r.east)}
			\prbd{(r.east) ++(2.5*\prad,0)}
			\pbdi{(r.east) ++(5.0*\prad,0)}
			\pext{(r.east) ++(7.5*\prad,0)}
			\path (r.east) ++(9.5*\prad,0) node (rend) {$\}$};
		\node (rp) at (0,1) [anchor=north west] {$R' = \{$};
			\prbd{(rp.east)}
			\pbdi{(rp.east) ++(2.5*\prad,0)}
			\path (rp.east) ++(4.5*\prad,0) node {$\}$};
		\node (rr) at (0,2) [anchor=north west] {$\dP_R(R') = \{$};
			\prbd{(rr.east)}
			\path (rr.east) ++(2.0*\prad,0) node {$\}$};
		\node (re) at (0,3) [anchor=north west]
			{$\{v''\}=R'' \setminus R'=\{$};
			\pext{(re.east)}
			\path (re.east) ++(2.0*\prad,0) node (reend){$\}$};
		\path (r.north west) +(135:0.125in) coordinate (tl);
		\path (re.south west) +(225:0.125in) coordinate (bl);
		\path (reend.south east) +(315:0.125in) coordinate (br);
		\path (br |- tl) coordinate (tr);
		\draw (tl) -- (tr) -- (br) -- (bl) -- cycle;
	\end{scope}
\end{tikzpicture}
\caption{%
An example of the sets relevant to the proof of Lemma~\ref{lem_18}.
On the left is a $5 \times 5$ subset of $\Z^2$,
with points decorated according to the key on the right.%
}
\label{f_lem18}
\end{figure}

\begin{proof}
Consider the ``relative boundary''
$\dP_R(R') := \{v' \in R' \colon \exists v \in R \setminus R', v \sim v'\}$,
i.e.\ the points in $R'$ that are directly adjacent to $R$.
As we shall see below, these are the only essentially relevant points of $R'$.
Indeed, we split the proof into two cases, depending on the restrictions
$h_{R'}|_{\dP_R(R')}$ and $\tlh_{R'}|_{\dP_R(R')}$ on $\dP_R(R')$.
The first case is the easier of the two.
In the first case, the two height functions agree on $\dP_R(R')$.
In the second case, there is a strict inequality
$h_{R'}(v) < \tlh_{R'}(v)$ for at least one point $v \in \dP_R(R')$.

Case 1: Assume first that $h_{R'}(v) = h_{R'}(v)$ for all $v \in \dP_R(R')$.
Since $\dP_R(R')$ may be a proper subset of $R'$,
this does not imply that $h_{R'} = \tlh_{R'}$.
(Although in the case where $\dP_R(R')$, then this case does indeed reduce
to the trivial assertion that $\muMRh$ is stochastically dominated by itself.)
However, it does hold that $h_{R'}$ and $\tlh_{R'}$ have the same extensions
to $R \setminus R'$; to be very precise,
\begin{equation} \label{e_set_id}
	\bigl\{ h_R \restr{R \setminus R'} \colon h_R \in M(R; h_{R'}) \bigr\}
	=
	\bigl\{ \tlh_R \restr{R \setminus R'} \colon
		\tlh_R \in M(R; \tlh_{R'}) \bigr\}.
\end{equation}

For clarity, let us repeat the above paragraph
in the context of Figure~\ref{f_lem18}.
Case 1 of the proof concerns data $h_{R'}$ and $\tlh_{R'}$ that agree
on the solid black circle region (i.e.\ $\dP_R(R')$),
though they may differ on the solid black square points
(i.e. $R' \setminus \dP_R(R')$).
By definition of $\dP_R(R')$, the solid black circle points surround
the white circle region (i.e.\ $R \setminus R'$),
at least relative to the domain $R$.
(Often we will assume that $\dP R \subset R'$,
but that assumption isn't necessary here, and it does not hold in the figure.)
Since $h_{R'}$ and $\tlh_{R'}$ agree on $\dP_R(R')$,
they have the same extensions to the white circle region,
in the sense of~\eqref{e_set_id}.


An easy calculation shows that for $h_R \in M(R; h_{R'})$,
\begin{equation} \label{e_rel_compl}
	\muMRh(h_R, \om)
	= \frac{
		\exp \bigl( H^+_{R \setminus R'}(h_R, \omega)
			\bigr)
	}{
		\sum_{f_R \in M(R; h_{R'})}
			\exp \bigl( H^+_{R \setminus R'}(f_R, \omega) \bigr)
	} ,
\end{equation}
where we recall that $H^+_{R \setminus R'}$ denotes the Hamiltonian on domain
$R \setminus R'$, including the edges that cross between $R \setminus R'$
and $\dP_R(R')$.
Therefore in particular that the right-hand expression~\eqref{e_rel_compl}
depends only on the values of the extension $h_R$
restricted $(R \setminus R') \cup \dP_R(R')$.
The same is true of extensions $\tlh_R \in M(R; \tlh_{R'})$.
As such, the obvious bijection between the two sets in~\eqref{e_set_id}
is measure-preserving in both directions.
The existence of a coupling~$\pi$ satisfying the claims of the lemma
follows immediately.

Case 2: Assume instead that $h_{R'}(v) < \tlh_{R'}(v)$ for some $v \in R'$
adjacent to a vertex $v'' \in R \setminus R'$.
In Figure~\ref{f_lem18}, $v''$ is the white circle marked with an ``X,''
and $v$ might be either of the adjacent solid black circles.
Let $R'' = R' \cup \{v''\}$.
We proceed by induction on the cardinality of $R \setminus R'$.
The induction hypothesis states that given any height functions
$h_{R''}, \tlh_{R''} \in M(R'')$ such that $h_{R''} \le \tlh_{R''}$
and such that both $M(R; h_{R''})$ and $M(R; \tlh_{R''})$ are nonempty,
the measure $\mu_{M(R; h_{R''})}$ is stochastically dominated
by $\mu_{M(R; \tlh_{R''})}$.
Note that the base case of the induction occurs when $R \setminus R' = \{v''\}$
has cardinality $1$, and so $R'' = R$; the lemma is trivial in this case.

So let us extend induction hypothesis from $R''$ to $R'$.
From the hypotheses of the lemma, each of $h_{R'}$ and $\tlh_{R'}$
admits at least one extension to $R$.
Therefore each admits at least one extension to $R''$
that in turn admits an extension to $R$.
On the other hand since $R'' \setminus R' = \{v''\}$
is a set of cardinality $1$,
each of $h_{R'}$ and $\tlh_{R'}$ admits at most two extensions to $R'$.
Formally, let $h_{R''}^+$ and $h_{R''}^-$ denote
the two possible extensions of $h_{R'}$ to $R''$,
where $h_{R''}^\pm(v'') = h_{R'}(v) \pm 1$.
Below we will address the possibility that one or the other of these
putative extensions does not exist.
Likewise, let $\tlh_{R''}^\pm \in M(R''; \tlh_{R'})$ denote the two extensions
of $\tlh_{R'}$ to $R''$, subject to the possibility that one or the other of the
two extensions may not exist.

By conditioning on the height value at $v''$,
we see that
\begin{equation} \label{e_mix} \begin{aligned}
	\mu_{M(R; h_{R'})}
	&= p^+ \mu_{M(R; h_{R''}^+)} + p^- \mu_{M(R; h_{R''}^-)}
	\quad\text{and} \\
	\mu_{M(R; \tlh_{R'})}
	&= \tlp^+ \mu_{M(R; \tlh_{R''}^+)} + \tlp^- \mu_{M(R; \tlh_{R''}^-)} ,
\end{aligned} \end{equation}
where
\begin{equation} \begin{aligned}
	p^\pm &:= \mu_{M(R; h_{R'})}(\{h_R(v'')
		= h_{R'}(v) \pm 1\}) \in [0,1]
	\quad\text{and} \\
	\tlp^\pm &:= \mu_{M(R; \tlh_{R'})}(\{\tlh_R(v'')
		= \tlh_{R'}(v) \pm 1\}) \in [0,1] .
\end{aligned} \end{equation}
This addresses the issue noted above, about the possibility that one (but not
both) of $h_{R''}^\pm$ may not exist; if so, the corresponding $p^\pm$ term
is $0$, and the other $p^\mp$ term is $1$.
By parity considerations, it must hold (assuming that the various extensions
exist), that
\[
	h_{R''}^-(v'') < h_{R''}^+(v'') = h_{R'}(v) + 1
	\le \tlh_{R'}(v) - 1 = \tlh_{R''}^-(v'') < \tlh_{R''}^+(v'') .
\]

By (up to) four applications of the induction hypothesis,
we conclude that each of the measures $\mu_{M(R; h_{R''}^\pm)}$
is stochastically dominated by each of the measures
$\mu_{M(R; \tlh_{R''}^\pm)}$.
Since all the measures are probability measures
and since all four of $p^\pm, \tlp^\pm$ are nonnegative,
the identities~\eqref{e_mix} implies that
$\mu_{M(R; h_{R'})}$ is stochastically dominated by $\mu_{M(R; \tlh_{R'})}$.
\end{proof}

We will make use of stochastic dominance via expectations,
as captured in the following corollary.

\begin{corollary}[cf.\ {\cite[Corollary 19]{CEP96}}] \label{cor_19}
Let $R' \subset R$, let~$v \in R \setminus R'$,
and let $h_{R'}, \tlh_{R'} \in M(R')$ with $h_{R'} \le \tlh_{R'} + 2$.
Let~$h_R(v)$ and~$\tlh_R(v)$ denote the $\Z$-valued random variables obtained
by sampling~$h_R$ from~$\muMRh \circ \P$ and~$\tlh_R$ from~$\muMRth \circ \P$
and evaluating the respective height functions at~$v$.
Then
\[
	E_{\muMRh \circ \P} [ h_R(v) ]
	\le E_{\muMRth \circ \P} [ \tlh_R(v) ] + 2 .
\]
\end{corollary}

\begin{remark}
Notice that unlike the stochastic monotonicity result of Lemma~\ref{lem_18},
which is almost sure in~$\om$,
the corollary above requires an expectation over the law~$\P$ of~$\om$.
This is a substantial difference from~\cite{CEP96}
caused by the random potential.
Indeed, the requirement arises from the fact that~$h_{R'}$ and~$\wt h_{R'}+2$
``see'' a different part of the random potential~$\om$.
Since~$\mu$ is a shift-invariant Gibbs measure,
we can average out this height shift by annealing over the random potential.
\end{remark}

\begin{proof}
By Lemma~\ref{lem_18}, for each fixed~$\om$
we have~$\muMRh(\cdot, \om) \overset{\textnormal{law}}\le
	\mu_{M(R; \tlh_{R'}+2)}(\cdot, \om)$,
so
\[
	E_{\muMRh(\cdot, \om)} [h_R(v)] \le
	E_{\mu_{M(R; \tlh_{R'}+2)}(\cdot, \om)} [\tlh_R(v)],
	\quad \text{for a.e.~$\om$} .
\]

We will transfer the height shift from the ``$\tlh_{R'}+2$''
into the random potential~$\om$ and into the height function
inside the expectation.
Indeed, from the definition of the Hamiltonian, we have for $R \subset \Z^d$
and~$f_R \in M(R)$ that
\begin{equation}
	H^\circ_R(f_R + 2, \om) = H^\circ_R(f_R, \tau_2 \om) ,
\end{equation}
where $\tau_2: \Omega \to \Omega$ is defined by
$(\tau_2(\om))_{x,x+1} = \om_{x+2,x+3}$ for all $x \in \Z$.
A straightforward calculation (see Appendix~\ref{app_e_tr})
establishes that, for any $h_R \in M(R)$,
\begin{equation} \label{e_tr}
	\mu_{M(R; \tlh_{R'}+2)}(h_R, \om)
	= \muMRth(h_R - 2, \tau_2 \om) .
\end{equation}

By change of variables,
\[
	E_{\muMRh(\cdot, \om)} [h_R(v)] \le
	E_{\muMRth(\cdot, \tau_2 \om)} [\tlh_R(v) + 2],
	\quad \text{for a.e.~$\om$} .
\]

Take expectations with respect to~$\P$.
Under the expectation the shift~$\tau_2$ vanishes, by ergodicity.
The result follows by construction of the measures
$\muMRh \circ \P$ and~$\muMRth \circ \P$;
cf.~equation~\eqref{e_cond_exp}.
\end{proof}

Now we are prepared to prove a limited version of the concentration inequality,
where we are concerned with only a single point $v \in R$.
The key to the proof is the monotonicity of Corollary~\ref{cor_19}.
We translate this into an inductive bound on martingale differences:
each time we take a ``step'' starting at the boundary $\dP R$ and ``walking''
towards $v$, the two possible extensions at that step differ by at most $2$.
Then we use the Azuma--Hoeffding inequality to establish the probability bound.
From this point on the proof is standard,
following closely to the methods used in~\cite{CEP96} and other works.

\begin{lemma}[Auxiliary concentration inequality
cf.\ {\cite[Theorem 21]{CEP96}}]
\label{lem_conc}
Let $h_\dR \in M(\dR)$ and let $v \in R$ be such that there is a path
$x_0 \in \dR, x_1, \dotsc, x_{l-1} = v$ of length $l$
with $x_i \sim x_{i-1}$ for $i=1, \dotsc, l-1$.
Then for any $c > 0$,
\[
	\muMRdR \circ \P \Bigl( \Bigl\{
		h_R \in M(R; h_\dR) \colon
		\abs[\big]{h_R(v) - E_{\muMRdR}[h_R(v)]} > l c
	\Bigr\} \Bigr) < 2e^{-l c^2 / 2} .
\]
\end{lemma}

\begin{proof}
For $k=1,\dotsc,l$, define $\sigma$-algebras
$\FF_k := \sigma(h_R \mapsto h_R(x_i), 0 \le i < k)
	\subset \PP(M(R; h_{\dP R}))$
and define a martingale $M_k := E[h_R(v)|\FF_k]$,
where $E[\,\cdot\,]$ denotes the expectation
with respect to the measure~$\muMRh \circ \P$.
Note that $M_1 = E[h_R(v)]$
and that $M_l = h_R(v)$.

We claim that for each $k=1, \dotsc, l-1$,
the martingale difference~$|M_{k+1}-M_k|$
is less than or equal to $2$ almost surely.
To this end, fix $k$
and condition on~$h(x_i) = z_i \in \Z$ for~$i=0,\dotsc,k-1$.
To avoid events of probability zero,
assume that~$z_0,\dotsc,z_{k-1}$ are such that
there exists at least one extension in $M(h_R; h_{\dP R})$
with $h_R(x_i) = z_i$ for each $i$;
by hypothesis~$M(h_R; h_{\dP R})$ is nonempty,
so at least one such assignment of heights~$z_i$ exists.

Having fixed these height values,
there are at most two assignments
of the height value $z_k := h_R(x_k)$
which admit further extensions in~$M(R; h_{\dP R})$:
namely, $z_k = z_{k-1} \pm 1$.
Therefore the martingale $M_{k+1} = E[h_R(v)|\FF_{k+1}]$
takes at most two distinct values
conditioned on~$\{h_R(x_i)=z_i, \, i=0,\dotsc,k-1\}$.
Because the (at most) two possible values of $h_R(x_k)$ differ by at most $2$,
and because the height values at $x_0, \dotsc, x_{k-1}$ have been fixed,
Corollary~\ref{cor_19} applied with $R' = \{x_0,\dotsc,x_k\}$ implies that
the (at most) two distinct values of $M_{k+1}$ differ by at most $2$.
Since $M_k = E[M_{k+1}|\FF_k]$ is the weighted average of these (at most)
two values of $M_{k+1}$, it follows that $\abs{M_{k+1} - M_k} \le 2$.
The conclusion follows immediately from the Azuma--Hoeffding inequality.
\end{proof}

Now we are prepared to prove the main result, i.e.\ the concentration
inequality.

\begin{proof}[Proof of Theorem~\ref{thm_conc}]
By the union bound and Lemma~\ref{lem_conc},
\begin{equation}
	\mu_n \circ \P \Bigl( \max_{v \in R_n} \:
		\abs[\big]{ h_{R_n}(v) - E_{\mu_n}(h_{R_n}(v)) }
			\ge c \sqrt{n} \Bigr)
	\le \sum_{v \in R_n} \mu_n \circ \P \Bigl( \,
		\abs[\big]{ h_{R_n}(v) - E_{\mu_n}(h_{R_n}(v)) }
			\ge c \sqrt{n} \Bigr) .
\end{equation}

For each $v \in R_n$, apply Lemma~\ref{lem_conc}
with the path length parameter $l_v$ chosen as small as possible
and with parameter $c_v$ chosen
such that $l_v c_v = c \sqrt{n}$.
Recall that by hypothesis the diameter of $R_n$ is at most $An$,
so $l \le An/2$.
It follows that
\[
	l_v c_v^2 = \frac{(l_v c_v)^2}{l_v} = \frac{c^2 n}{l_v}
	\ge \frac{2 c^2 n}A .
\]
Using also the hypotheses that $|R_n| \le Bn^m$, we have
\begin{equation} \label{e_unif} \begin{aligned}
	\hskip3em&\hskip-3em
	\mu_n \circ \P \biggl( \max_{v \in R_n} \:
		\abs[\big]{ h_{R_n}(v) - E_\mu(h_{R_n}(v)) } \ge c \sqrt{n}
	\biggr) \\
	&\le \sum_{v \in R_n} \mu_n \circ \P \biggl(
		\abs[\big]{ h_{R_n}(v) - E_{\mu_n} \bigl( h_{R_n}(v) \bigr) }
		\ge l_v c_v \biggr) \\
	&\le \sum_{v \in R_n} 2 e^{-l_v c_v^2 / 2} \\
	&\le 2 |R_n| e^{- c^2 n / A} .
\end{aligned} \end{equation}
\end{proof}

\appendix

\section{Proof of Eq.~(\ref{e_tr}) from Corollary~\ref{cor_19}}
\label{app_e_tr}

The following equation was used in the proof of Corollary~\ref{cor_19} above.
The proof that this equation holds has been moved here in an appendix,
since it requires several lines, and it is not enlightening or crucial enough to
justify taking space in the body of the proof.
The equation in question is (cf.~\eqref{e_tr} above):

 \[
	\mu_{M(R; \tlh_{R'})+2}(h_R, \om)
	= \muMRth(h_R - 2, \tau_2 \om) .
\]

Indeed:
\begin{align*}
	\mu_{M(R; \tlh_{R'}+2)}(h_R, \om)
	&= \frac{ \exp \bigl( H^\circ_R(h_R, \om) \bigr) }{
		\sum_{f_R \in M(R; h_{R'}+2)} \exp \bigl(
			H^\circ_R(f_R, \om) \bigr)} \\
	&= \frac{ \exp \bigl( H^\circ_R(h_R, \om) \bigr) }{
		\sum_{f_R \in M(R; h_{R'})} \exp \bigl(
			H^\circ_R(f_R + 2, \om) \bigr)} \\
	&= \frac{ \exp \bigl( H^\circ_R(h_R, \om) \bigr) }{
		\sum_{f_R \in M(R; h_{R'})} \exp \bigl(
			H^\circ_R(f_R, \tau_2 \om) \bigr)} \\
	&= \frac{ \exp \bigl( H^\circ_R(h_R - 2, \tau_2 \om) \bigr) }{
		\sum_{f_R \in M(R; h_{R'})} \exp \bigl(
			H^\circ_R(f_R, \tau_2 \om) \bigr)} \\
	&= \muMRth(h_R - 2, \tau_2 \om) .
\end{align*}
\hfill\qed

\bibliographystyle{alpha}
\bibliography{bib}

\end{document}